\documentclass[12pt,a4paper,reqno]{amsart}
\usepackage[UKenglish]{babel}
\usepackage[T1]{fontenc}
\usepackage{palatino}
\usepackage{amsmath}
\usepackage{mathabx}
\usepackage{amsthm}
\usepackage{xcolor}
\usepackage{enumerate}
\usepackage{comment}
\usepackage[colorlinks = true, linkcolor={red!80!black}, citecolor = {blue!60!black}]{hyperref}

\newcommand{\N}{\mathbb{N}}
\newcommand{\Z}{\mathbb{Z}}
\newcommand{\R}{\mathbb{R}}
\newcommand{\e}{\epsilon}

\theoremstyle{plain}
\newtheorem{thm}{Theorem}
\newtheorem{cor}[thm]{Corollary}
\newtheorem{prop}[thm]{Proposition}
\newtheorem{lem}[thm]{Lemma}

\theoremstyle{definition}
\newtheorem{definition}[thm]{Definition}
\newtheorem{remark}[thm]{Remark}
\newtheorem{ex}[thm]{Example}

\numberwithin{equation}{section}
\numberwithin{thm}{section}

\author{Tuomas Orponen and Pablo Shmerkin}
\title[Furstenberg type estimates]{Furstenberg-type estimates under mild non-concentration assumptions}
\address{Department of Mathematics and Statistics\\ University of Jyv\"askyl\"a,
P.O. Box 35 (MaD)\\
FI-40014 University of Jyv\"askyl\"a\\
Finland} \email{tuomas.t.orponen@jyu.fi}
\address{Department of Mathematics\\
The University of British Columbia\\
1984 Mathematics Road, Vancouver, BC\\
Canada} \email{pshmerkin@math.ubc.ca}
\date{\today}
\thanks{T.O. is supported by the European Research Council (ERC) under the European Union's Horizon Europe research and innovation programme (grant agreement No 101087499), and by the Research Council of Finland via the project \emph{Approximate incidence geometry}, grant no. 355453.}
\thanks{P.S. is supported by an NSERC Discovery Grant.}

\begin{document}

\begin{abstract} 
We prove sharp $\delta$-discretised versions of some variants of the Furstenberg set problem under weaker or different non-concentration assumptions compared to previous works. 
 \end{abstract}

\maketitle

\tableofcontents

\section{Introduction and main results}

The study of projections and more general discretised incidence problems has seen many significant developments in recent years. Notably, the Furstenberg set conjecture was recently resolved by Ren and Wang \cite{RenWang23}, building up on our previous work \cite{OS23}. We next state the $\delta$-discretised version of this estimate. The original conjecture, due to Wolff \cite{Wolff99}, was stated in terms of Hausdorff dimension, but the $\delta$-discretised formulation is more powerful in applications, such as the ones in this paper. We refer to Section \ref{sec:preliminaries} for the definitions of the discretised objects (for example Frostman sets and configurations) appearing in the statement. Here we just informally mention that a \emph{$(\delta,s,C,M)$-configuration} $(\mathcal{P},\mathcal{T})$ is a pair of $\delta$-squares $\mathcal{P}$ and $\delta$-tubes $\mathcal{T}$ such that each square $p \in \mathcal{P}$ intersects all tubes in an $s$-dimensional sub-family $\mathcal{T}(p) \subset \mathcal{T}$ of cardinality $|\mathcal{T}(p)| = M$.

\begin{thm}[Furstenberg set estimate]\label{thm:FurstIntro} 
  Let $s \in (0,1]$, $t \in [0,2]$, and let
  \begin{displaymath} 
    \gamma < \min\left\{t,\tfrac{s + t}{2},1\right\}. 
  \end{displaymath} 
  Then, there exist $\epsilon = \epsilon(s,t,\gamma) > 0$ and $\delta_{0} = \delta_{0}(s,t,\gamma) > 0$ such that the following holds for all $\delta \in (0,\delta_{0}]$. Let $(\mathcal{P},\mathcal{T})$ be a $(\delta,s,\delta^{-\epsilon},M)$-nice configuration, where $\mathcal{P}$ is a Frostman $(\delta,t,\delta^{-\epsilon})$-set. Then,
\begin{displaymath} 
  |\mathcal{T}| \geq M\delta^{-\gamma}. 
\end{displaymath} 
Moreover, $\delta_0,\epsilon$ can be chosen uniform over all $(s,t,\gamma)$ in a fixed compact subset of $(0,1] \times [0,2] \times (0,\min\{t,\tfrac{s + t}{2},1\}))$. 
\end{thm}

Theorem \ref{thm:FurstIntro} is \cite[Theorem 4.1]{RenWang23}. Many variants of Theorem \ref{thm:FurstIntro} have been achieved, with important applications to the restriction conjecture and discretised sum-product problems, among others. See \cite{DemeterORegan25, DemeterWang25, WangWu24,WangWu26}. Each of these variants involves some relaxation of the non-concentration assumptions on the tubes $\mathcal{T}$ (often at the cost of some additional assumptions on the set $\mathcal{P}$). 

The main goal of this paper is to obtain versions of Theorem \ref{thm:FurstIntro} which are valid under weaker non-concentration assumptions on the set $\mathcal{P}$, but keeping a strong non-concentration assumption on the tube family $\mathcal{T}$. The first main result is Theorem \ref{thm:Furstenberg-minimal-non-concentration} below. The point of this result is to obtain an improvement over the "trivial" bound $|\mathcal{T}| \gtrapprox M|\mathcal{P}|^{1/2}$ (see for example \cite[Proposition 2.9]{MR4912925} for a precise statement) under the hypothesis that $\mathcal{P}$ is not concentrated in a single square of (minimal) side-length $\delta|\mathcal{P}|^{1/2}$.
\begin{thm}\label{thm:Furstenberg-minimal-non-concentration}
Fix
\begin{displaymath}
  s\in (0,1], \quad t\in [0,2], \quad u \in (0,\min\{t,2-t\}],
\end{displaymath}
and
\begin{displaymath} 
  0 < \zeta < \frac{t}{2} + \frac{s \cdot u}{2}.
\end{displaymath} 
Then, there exist $\e=\e(\zeta,s,t,u) > 0$ and $\delta_{0} = \delta_{0}(\zeta,s,t,u) > 0$ such that the following holds for all $\delta \in 2^{-\N} \cap (0,\delta_{0}]$. Let $(\mathcal{P},\mathcal{T})$ be a $(\delta,s,\delta^{-\e},M)$-nice configuration. Assume that  $|\mathcal{P}|=\delta^{-t}$ and
\begin{equation} \label{eq:single-scale-non-conc}
|\mathcal{P} \cap Q| \le \delta^u |\mathcal{P}|, \quad Q\in\mathcal{D}_{\delta|\mathcal{P}|^{1/2}}.
\end{equation}
Then
\begin{equation}\label{form95}
|\mathcal{T}| \ge M\cdot \delta^{-\zeta}. 
\end{equation}
\end{thm}

It is well known that discretised Furstenberg set estimates imply discretised projection estimates. In short, for an $s$-Frostman set $\Theta \subset S^{1}$ of directions, one considers for each $p\in\mathcal{P}$ the tubes $\mathcal{T}(p)$ through $p$ in directions $\theta\in\Theta$. Then a Furstenberg set estimate applied to the $(\delta,s,C,|\Theta|)$-nice configuration $(\mathcal{P},\bigcup_{p} \mathcal{T}(p))$ implies a projection estimate. With this argument, Theorem \ref{thm:Furstenberg-minimal-non-concentration} implies the following projection estimate under minimal non-concentration assumptions:
\begin{cor} \label{cor:projDeltaDiscretisedSingleScale} 
  Fix $s,t,u$, and $\zeta$ as in Theorem \ref{thm:Furstenberg-minimal-non-concentration}. Then,  there exist $\epsilon = \epsilon(\zeta,s,t,u) > 0$ and $\delta_{0} = \delta_{0}(\zeta,s,t,u) > 0$ such that the following holds for all $\delta \in 2^{-\N} \cap (0,\delta_{0}]$. 

Assume that $\mathcal{P} \subset \mathcal{D}_{\delta}$ has $|\mathcal{P}|_{\delta}=\delta^{-t}$, and satisfies \eqref{eq:single-scale-non-conc}. Let $\Theta \subset S^{1}$ be a non-empty Frostman $(\delta,s,\delta^{-\epsilon})$-set. Then, there exists $\theta \in \Theta$ such that
\begin{equation}\label{form126}
|\pi_{\theta}(\mathcal{P}')|_{\delta} \geq \delta^{-\zeta}, \qquad \mathcal{P}' \subset \mathcal{P}, \, |\mathcal{P}'| \geq \delta^{\epsilon}|\mathcal{P}|.
\end{equation}
Here $\pi_{\theta}(x) := x \cdot \theta$ for $x \in \R^{2}$.
\end{cor}

The single-scale non-concentration assumption \eqref{eq:single-scale-non-conc} first appeared in \cite{Shmerkin23}, where the second author proved a non-quantitative version of Corollary \ref{cor:projDeltaDiscretisedSingleScale}. We note that in Theorems \ref{thm:FurstIntro} and \ref{thm:Furstenberg-minimal-non-concentration}, and in Corollary \ref{cor:projDeltaDiscretisedSingleScale}, the cases $t \in \{0,2\}$ are easily checked to be trivial, but it is convenient to include them in the statement. 

Theorem \ref{thm:Furstenberg-minimal-non-concentration} and Corollary \ref{cor:projDeltaDiscretisedSingleScale} are sharp; this will be shown in \S \ref{subsec:sharpness1}.

We then move to the second main theorem of the paper, where the "minimal" non-concentration \eqref{eq:single-scale-non-conc} on $\mathcal{P}$ is upgraded to a \emph{Katz-Tao} non-concentration hypothesis (see Definition \ref{FrostmanKatzTao}).

\begin{thm}\label{thm1} Let $s \in (0,1]$, $t \in [s,2 - s]$, and $\eta > 0$. Then there exist $\epsilon = \epsilon(s,t,\eta) > 0$ and $\delta_{0} = \delta_{0}(s,t,\eta) > 0$ such that the following holds for all $\delta \in (0,\delta_{0}]$. Let $\mathcal{P} \subset \mathcal{D}_{\delta}$ be a Katz-Tao $(\delta,t,\delta^{-\epsilon})$-set, and let $(\mathcal{P},\mathcal{T})$ be a $(\delta,s,\delta^{-\epsilon},M)$-nice configuration. Then,
\begin{displaymath} 
  |\mathcal{T}| \geq \delta^{\eta}M|\mathcal{P}|^{(s + t)/(2t)}.
\end{displaymath} 
\end{thm}
Theorem \ref{thm1} is also sharp; see \S \ref{subsec:sharpness2}. The sharpness examples have the form $P = A \times A$, where $A \subset \delta \Z \cap [0,1]$ is a Katz-Tao $(\delta,t/2)$-set. Is Theorem \ref{thm1} also sharp for product sets $P = A \times B$, where $|A| \neq |B|$? Not always: even the trivial bound $|\pi_{\theta}(P)|_{\delta} \gtrsim |A|$ may sometimes beat the bound $(|A||B|)^{(1 + s)/2}$ suggested by applying Theorem \ref{thm1} directly to $P=A\times B$. Motivated by this observation, we present a variant of Theorem \ref{thm1} for (generally) asymmetric product sets. This comes at negligible additional cost, since the proof of Theorem \ref{thm2} is almost the same as the proof of Theorem \ref{thm1}.

\begin{thm}\label{thm2} Let $s \in (0,1]$, let $\alpha,\beta \in (0,1]$ with $\alpha + \beta \in [s,2 - s]$, and let $\eta > 0$. Then there exist $\epsilon = \epsilon(s,\alpha,\beta,\eta) > 0$ and $\delta_{0} = \delta_{0}(s,\alpha,\beta,\eta) > 0$ such that the following holds for all $\delta \in (0,\delta_{0}]$. 
\begin{itemize}
\item Let $A \subset \delta\Z \cap [0,1]$ be a Katz-Tao $(\delta,\alpha,\delta^{-\epsilon})$-set.
\item Let $B \subset \delta \Z \cap [0,1]$ be a Katz-Tao $(\delta,\beta,\delta^{-\epsilon})$-set.
\end{itemize}
Assume that $(A \times B,\mathcal{T})$ is a $(\delta,s,\delta^{-\epsilon},M)$-nice configuration, and $|A|^{\beta} \geq |B|^{\alpha}$. Then,
\begin{equation}\label{form14}
|\mathcal{T}| \geq \delta^{\eta} M\begin{cases} |A||B|^{(\beta + s - \alpha)/(2\beta)}, & \alpha \leq s, \\ |A|^{(\alpha + s)/(2\alpha)}|B|^{1/2}, & \alpha \geq s. \end{cases}
\end{equation}  
%\begin{displaymath} |\mathcal{T}| \geq \delta^{\eta}M|A|^{(\alpha + \theta s)/(2\alpha)}|B|^{(\beta + (1 - \theta)s)/(2\beta)} \end{displaymath} 
%for every $\theta \in [0,1]$ satisfying $\max\{\theta s/\alpha,(1 - \theta) s/\beta\} \leq 1$.
\end{thm}

%\begin{remark} One can optimise the choice of "$\theta$" in Theorem \ref{thm2} in terms of the data $|A|,|B|,\alpha,\beta$. To do this, first express the lower bound for $|\mathcal{T}|$ in the form
%\begin{displaymath} M(|A||B|)^{1/2}|B|^{s/(2\beta)} \cdot (|A|^{s/(2\alpha)}|B|^{-s/(2\beta)})^{\theta}. \end{displaymath}  
%Now, if $|A|^{\beta} \leq |B|^{\alpha}$, it is desirable to minimise $\theta$, and in the opposite case it is desirable to maximise $\theta$. One just needs to take into account the constraints $\theta \in [0,1]$ and $\max\{\theta s/\alpha,(1 - \theta)s/\beta\} \leq 1$. The conclusion is the following. Assume that $\alpha,\beta,s$ are as in Theorem \ref{thm2}, and $|A|^{\beta} \geq |B|^{\alpha}$. Then,
%\begin{equation}\label{form14} |\mathcal{T}| \geq \begin{cases} \delta^{\eta} M|A||B|^{(\beta + s - \alpha)/(2\beta)}, & \alpha \leq s, \\ \delta^{\eta}M|A|^{(\alpha + s)/(2\alpha)}|B|^{1/2}, & \alpha \geq s. \end{cases} \end{equation}  
%\end{remark} 

%\begin{remark} The condition $\max\{\theta s/\alpha,(1 - \theta)s/\beta\} \leq 1$ is satisfied (e.g.) for $\theta \in \{\alpha/(\alpha + \beta),\beta/(\alpha + \beta)\} \in [0,1]$ thanks to the hypothesis $s \leq \alpha + \beta$.
%\end{remark}

\begin{remark}\label{rem1} %The bounds \eqref{form14} remain valid if $|A|^{\beta} < |B|^{\alpha}$, but then sharper bounds can be obtained by reversing the roles of $A$ and $B$. 
We only examined the sharpness of Theorem \ref{thm2} in the case $\alpha = \beta$. Then, for $|A| \geq |B|$, \eqref{form14} and $2\alpha \in [s,2 - s]$, the bound \eqref{form14} simplifies to
\begin{equation}\label{form23}
|\mathcal{T}| \geq \delta^{\eta}M\begin{cases} |A||B|^{s/(2\alpha)}, & \alpha \leq s, \\ |A|^{(\alpha + s)/(2\alpha)}|B|^{1/2}, & \alpha \geq s. \end{cases}
\end{equation} 
It turns out that the first bound is always sharp, and the second bound is sharp whenever $|A|^{\alpha - s} \leq |B|^{\alpha}$. In fact, for $|A|^{\alpha - s} = |B|^{\alpha}$ the second bound simplifies to $|\mathcal{T}| \geq \delta^{\eta}M|A|$. In the range $|B|^{\alpha} < |A|^{\alpha - s}$, the "trivial bound" $|\mathcal{T}| \gtrapprox M|A|$ is the sharp bound (and is stronger than the one given by Theorem \ref{thm2}).

We will sketch sharpness examples to \eqref{form23} in Section \ref{rem:sharpness-thm2}.
\end{remark}

Theorems \ref{thm1}--\ref{thm2} should be contrasted with the recent results of H. Wang and S. Wu \cite{WangWu24,WangWu26}. The difference is that the non-concentration condition of Theorem \ref{thm1} for the tube families $\mathcal{T}(p)$ incident to the squares $p \in \mathcal{P}$ is significantly stronger than in \cite{WangWu24,WangWu26}. On the other hand the conclusions are also stronger in cases where $M$ or $|\mathcal{P}|$ are small. Let us give one concrete comparison: in the case $t = 1$ the main theorem of \cite{WangWu24} applied to the setting of Theorem \ref{thm1} yields the lower bound $|\mathcal{T}| \gtrapprox M^{3/2}\delta^{1/2}|\mathcal{P}|$. The lower bound is smaller (weaker) than the one in Theorem \ref{thm1} whenever $M|\mathcal{P}|^{1 - s} < \delta^{-1}$. In fact, even the "trivial" bound $|\mathcal{T}| \gtrapprox M|\mathcal{P}|^{1/2}$ beats \cite{WangWu24} when $M|\mathcal{P}| < \delta^{-1}$.

The proofs of all our theorems follow the ``combinatorial scale decomposition'' approach. Applying Lipschitz function combinatorial decompositions from \cite{Shmerkin23,ShmerkinWang25}, we are able to locate scale blocks on which Theorem \ref{thm:FurstIntro} can be applied. The details of the multiscale decomposition are new. The different scale blocks are combined via a multi-scale incidence estimate from \cite{OS23}. This background is recalled in Section \ref{sec:preliminaries}, and the proofs of the main theorems are given in Sections \ref{sec:proofs1} and \ref{sec:proofs2}. The sharpness examples are discussed in Section \ref{sec:sharpness}. 

\begin{remark} 
  We comment on the history of this project. Initially, a non-sharp version of Theorem \ref{thm:Furstenberg-minimal-non-concentration} (with $su/4$ instead of $su/2$) was obtained in the first \emph{arXiv} version of our article \cite{OrponenShmerkin26} in 2023. This estimate was based on the ``sticky'' Furstenberg estimate obtained in \cite{OrponenShmerkin26}. A few months later, K.~Ren and H.~Wang proved the sharp Furstenberg set estimate in Theorem \ref{thm:FurstIntro}. We realised that by applying their result we could obtain the sharp version of Theorem \ref{thm:Furstenberg-minimal-non-concentration}, stated above. Since the sharp version Theorem \ref{thm:Furstenberg-minimal-non-concentration} relies on the sharp Furstenberg set estimate that appeared after our preprint, and following a referee suggestion, we decided to split off Theorem \ref{thm:Furstenberg-minimal-non-concentration} from \cite{OrponenShmerkin26} into the present article. In addition to strengthening the result, we have also added and corrected details in the proofs with respect to the \emph{arXiv} version of \cite{OrponenShmerkin26}. Theorems \ref{thm1} and \ref{thm2} appear here for the first time.
\end{remark}

\section{Preliminaries and notation}

\label{sec:preliminaries}

\subsection{Dyadic scales, tubes, and nice configurations}

We fix dyadic scales $\delta,\Delta \in 2^{-\N}$ and follow the conventions in \cite[Section 2 and Appendix A]{OrponenShmerkin26}. 

For $d \in \N$ and $A\subset \R^d$, write
\begin{displaymath}
\mathcal{D}_{\delta}(A):=\{p\in \delta\Z^d+[0,\delta)^d:\, p\cap A\neq\emptyset\}
\end{displaymath}
for the $\delta$-dyadic squares hitting $A$.  We abbreviate $\mathcal{D}_{\delta}:=\mathcal{D}_{\delta}([-1,1)^d)$.

We denote $|A|_{\delta}=|\mathcal{D}_{\delta}(A)|$. We extend this to non-dyadic $\delta$ by setting $|A|_{\delta}:=|A|_{\delta'}$ where $\delta'=\min\{2^{-n}:\, 2^{-n}\le \delta\}$. Up to a multiplicative constant, $|A|_{\delta}$ equals the $\delta$-covering number of $A$.

\begin{definition}[Frostman and Katz-Tao $(\delta,s,C)$-sets]\label{FrostmanKatzTao} Let $\delta \in 2^{-\N}$, $s \geq 0$ and $C > 0$. A set $\mathcal{P} \subset \mathcal{D}_{\delta}$ is called a \emph{Frostman $(\delta,s,C)$-set} if
\begin{displaymath}
|\mathcal{P} \cap Q| \leq Cr^{s}|\mathcal{P}|, \qquad Q \in \mathcal{D}_{r}, \, r \in 2^{-\N} \cap [\delta,1].
\end{displaymath}
Similarly, a set $\mathcal{P} \subset \mathcal{D}_{\delta}$ is called a \emph{Katz-Tao $(\delta,s,C)$-set} if
\begin{displaymath}
|\mathcal{P} \cap Q| \leq C(r/\delta)^{s}, \qquad Q \in \mathcal{D}_{r}, \, r \in 2^{-\N} \cap [\delta,1].
\end{displaymath}
If the constant "$C$" is absolute, we sometimes simply write "Frostman $(\delta,s)$-set" or "Katz-Tao $(\delta,s)$-set".    \end{definition}

We then proceed to define \emph{dyadic $\delta$-tubes}. For $(a,b)\in \R^2$, let
\begin{displaymath}
\mathbf{D}(a,b):=\{(x,y)\in \R^2:\, y=ax+b\}.
\end{displaymath}
This is the ``point-line duality'' map. If $p=[a,a+\delta)\times [b,b+\delta)\in\mathcal{D}_{\delta}([-1,1)\times\R)$, define the associated \emph{dyadic $\delta$-tube} by
\begin{displaymath}
\mathbf{D}(p):=\bigcup_{(a',b')\in p}\mathbf{D}(a',b').
\end{displaymath}
We set
\begin{displaymath}
\mathcal{T}^{\delta}:=\{\mathbf{D}(p):\, p\in \mathcal{D}_{\delta}([-1,1)\times\R)\}.
\end{displaymath}
For $T=\mathbf{D}([a,a+\delta)\times[b,b+\delta))\in\mathcal{T}^{\delta}$, define its \emph{slope} by $\sigma(T):=a$.

%A \emph{$(\delta,s,C)$-set} is a bounded set $A\subset\R^d$ such that
%\begin{displaymath}
%|A\cap B(x,r)|_{\delta} \leq C r^{s}|A|_{\delta}\quad\text{for all }x\in\R^d,\ r\in[\delta,1].
%\end{displaymath}

We identify subsets $\mathcal{P}\subset\mathcal{D}_{\delta}$ with the union of the squares in $\mathcal{P}$.

For a pair $(\mathcal{P},\mathcal{T})\subset \mathcal{D}_{\delta}\times\mathcal{T}^{\delta}$ and $p\in\mathcal{P}$, write
\begin{displaymath}
\mathcal{T}(p):=\{T\in\mathcal{T}:\, T\cap p\neq\emptyset\},\qquad
\sigma(\mathcal{T}(p)):=\{\sigma(T):\,T\in\mathcal{T}(p)\}.
\end{displaymath}
We say that $(\mathcal{P},\mathcal{T})$ is a \emph{$(\delta,s,C,M)$-nice configuration} if for every $p\in\mathcal{P}$,
\begin{displaymath}
C^{-1}M \le |\mathcal{T}(p)| \le CM,
\end{displaymath}
and $\sigma(\mathcal{T}(p))$ is a Frostman $(\delta,s,C)$-set in $[-1,1)$.

Finally, for $p\in \mathcal{D}_{\Delta}$ let $S_p:\R^2\to\R^2$ be the orientation-preserving homothety mapping $p$ onto $[0,1)^2$. For $\mathcal{P}\subset\mathcal{D}_{\delta}$ with $\delta\le \Delta$,
\begin{displaymath}
S_p(\mathcal{P}\cap p):=\{S_p(q):\,q\in\mathcal{P},\ q\subset p\}\subset\mathcal{D}_{\delta/\Delta}.
\end{displaymath}

We refer to \cite[Section 2]{OrponenShmerkin26} for more details on the above definitions.

\subsection{Uniform sets and branching functions}

\label{subsec:uniform-sets}

Assume that $\delta=\Delta^m$ with $\Delta=2^{-T}\in 2^{-\N}$, and let $\mathcal{P}\subset \mathcal{D}_{\delta}$. We say that $\mathcal{P}$ is \emph{$\{\Delta^j\}_{j=1}^m$-uniform} if there exist numbers $N_1,\ldots,N_m\ge 1$ such that for every $1\le j\le m$ and every $\mathbf{p}\in \mathcal{D}_{\Delta^{j-1}}(\mathcal{P})$,
\begin{displaymath}
\big|\mathcal{D}_{\Delta^j}(\mathcal{P})\cap \mathbf{p}\big|=N_j.
\end{displaymath}
In this case, the \emph{branching function} $\beta: [0,m]\to [0,2m]$ of $\mathcal{P}$ is defined by
\begin{displaymath}
\beta(j):=\frac{\log |\mathcal{P}|_{\Delta^j}}{T}=\frac{1}{T}\sum_{i=1}^j \log N_i,\qquad 0\le j\le m,
\end{displaymath}
and then linearly interpolating for non-integer values of the parameter. Note that $\beta$ is $2$-Lipschitz, non-decreasing, piecewise-linear, and satisfies $\beta(0)=0$ and $\beta(m)=\log |\mathcal{P}|/\log(1/\Delta)$.

The next standard pigeonholing lemma allows us to pass to a large uniform subset (see \cite[Appendix A]{OrponenShmerkin26}).
\begin{lem}\label{l:uniformization}
Let $\epsilon>0$. Then there exists $T_0=T_0(\epsilon)\in\N$ such that the following holds. If $\Delta=2^{-T}$ with $T\ge T_0$, $\delta=\Delta^m$, and $\mathcal{P}\subset\mathcal{D}_{\delta}$, then there exists a $\{\Delta^j\}_{j=1}^m$-uniform set $\mathcal{P}'\subset\mathcal{P}$ such that
\begin{displaymath}
|\mathcal{P}'|\ge \delta^{\epsilon}|\mathcal{P}|.
\end{displaymath}
In particular, one may arrange $T^{-1}\log(2T)\le \epsilon$.
\end{lem}

\subsection{Lipschitz decomposition lemma}

%Given a function $f:[a,b]\to\R$, we write
%\begin{displaymath}
%s_f(a,b):=\frac{f(b)-f(a)}{b-a}
%\end{displaymath}
%for the slope of the secant line through $(a,f(a))$ and $(b,f(b))$.

Given $f:[a,b]\to\R$ and numbers $\e>0,\sigma$, we say that $(f,a,b)$ is \emph{$(\sigma,\e)$-superlinear} if
\begin{displaymath}
f(x)\ge f(a)+\sigma(x-a)-\e(b-a)\qquad\text{for all }x\in[a,b].
\end{displaymath}
%Note that we do not assume that $\sigma=s_f(a,b)$. %In the case where we do have $\sigma=s_f(a,b)$, we simply say that $(f,a,b)$ is \emph{$\e$-superlinear}. We say that $(f,a,b)$ is \emph{$\e$-linear} if both $(f,a,b)$ and $(-f,a,b)$ are $\e$-superlinear. Equivalently,
%\begin{displaymath}
%\big|f(x)-L_{f,a,b}(x)\big|\le \e|b-a|\qquad\text{for all }x\in[a,b],
%\end{displaymath}
%where $L_{f,a,b}$ is the affine function agreeing with $f$ at $a$ and $b$.
The next lemma is \cite[Lemma 2.11]{OrponenShmerkin26} (and it follows directly by combining \cite[Lemmas 5.21--5.22]{ShmerkinWang25}).
\begin{lem}\label{l:combinatorial-weak}
For every $d\in\N$ and $\epsilon>0$ there is $\tau=\tau(d,\e)>0$ such that the following holds: for any non-decreasing $d$-Lipschitz function $f:[0,m]\to\R$ with $f(0)=0$ there exist sequences
\begin{align*}
0&=a_0 < a_1 <\cdots < a_{n} =m,\\
0&\le t_0 <t_1<\cdots <t_{n-1} \le d,
\end{align*}
such that:
\begin{enumerate}[(i)]
\item $a_{j+1}-a_j \ge \tau m$.
\item $(f,a_{j},a_{j+1})$ is $(t_j,0)$-superlinear.
\item $\sum_{j=0}^{n - 1} (a_{j+1}-a_j)t_j \ge f(m)-\e m$.
\end{enumerate}
\end{lem}

\begin{remark}\label{r:combinatorial-weak}
Note that if $F$ is the piecewise affine function with $F(0)=0$ and slope $t_j$ on $[a_j,a_{j+1}]$, then (ii) and (iii) imply
\begin{displaymath}
F(a_j)\le f(a_j)\le F(a_j)+\e m,\qquad j\in\{0,\ldots,n-1\},
\end{displaymath}
and consequently
\begin{displaymath}
|f(a_i)-f(a_j)-(F(a_i)-F(a_j))|\le \e m,\qquad i,j\in\{0,\ldots,n-1\}.
\end{displaymath}
\end{remark}

\subsection{Multiscale decomposition of nice configurations}

We now repeat \cite[Corollary 4.1]{ShmerkinWang25b} (itself based on \cite[Proposition 5.1]{OS23}). In the statement, $A\lessapprox_{\delta} B$ stands for $A\le C\log(1/\delta)^C B$ for some universal constant $C>0$, and likewise for $A\approx_{\delta} B$.
\begin{cor}\label{cor:ShWa} Fix $N \geq 2$ and a sequence $\{\Delta_{j}\}_{j = 0}^{N} \subset 2^{-\N}$ with
\begin{displaymath} 
    0 < \delta = \Delta_{N} < \Delta_{N - 1} < \ldots < \Delta_{1} < \Delta_{0} = 1. 
\end{displaymath}
Let $(\mathcal{P}_{0},\mathcal{T}) \subset \mathcal{D}_{\delta} \times \mathcal{T}^{\delta}$ be a $(\delta,s,C,M)$-nice configuration. Then, there exists a set $\mathcal{P} \subset \mathcal{P}_{0}$ such that the following properties hold:
\begin{itemize}
\item[(i)] $|\mathcal{D}_{\Delta_{j}}(\mathcal{P})| \approx_{\delta} |\mathcal{D}_{\Delta_{j}}(\mathcal{P}_{0})|$ and $|\mathcal{P} \cap \mathbf{p}| \approx_{\delta} |\mathcal{P}_{0} \cap \mathbf{p}|$, $1 \leq j \leq N$, $\mathbf{p} \in \mathcal{D}_{\Delta_{j}}(\mathcal{P})$.
\item[(ii)] For every $0 \leq j \leq N - 1$ and $\mathbf{p} \in \mathcal{D}_{\Delta_{j}}(\mathcal{P})$, there exist numbers
\begin{displaymath}
   C_{\mathbf{p}} \approx_{\delta} C \quad \text{and} \quad M_{\mathbf{p}} \geq 1, 
\end{displaymath}
and a family of tubes $\mathcal{T}_{\mathbf{p}} \subset \mathcal{T}^{\Delta_{j + 1}/\Delta_{j}}$ with the property that $(S_{\mathbf{p}}(\mathcal{P} \cap \mathbf{p}),\mathcal{T}_{\mathbf{p}})$ is a $(\Delta_{j + 1}/\Delta_{j},s,C_{\mathbf{p}},M_{\mathbf{p}})$-nice configuration.
\end{itemize}
Furthermore, the families $\mathcal{T}_{\mathbf{p}}$ can be chosen so that
\begin{displaymath} 
    \frac{|\mathcal{T}|}{M} \gtrapprox_{\delta} \prod_{j = 0}^{N - 1} \max_{\mathbf{p}_{j} \in \mathcal{D}_{\Delta_{j}}} \frac{|\mathcal{T}_{\mathbf{p}_{j}}|}{M_{\mathbf{p}_{j}}}. 
\end{displaymath}
All the constants implicit in the $\approx_{\delta}$ notation are allowed to depend on $N$.
\end{cor}

\section{Proof of Theorem \ref{thm:Furstenberg-minimal-non-concentration}} 

\label{sec:proofs1}

The idea of the proof of Theorem \ref{thm:Furstenberg-minimal-non-concentration} is to first apply to (the branching function of) $\mathcal{P}$ a multiscale decomposition provided by Lemma \ref{l:combinatorial-weak}. This procedure decomposes the scales between $\delta$ and $1$ into "blocks" where $\mathcal{P}$ looks roughly $t_{j}$-dimensional for an increasing sequence of exponents $t_{j} \in [0,2]$. On each "block" separately, we obtain a lower bound from Theorem \ref{thm:FurstIntro}, and eventually combine the estimates by applying Corollary \ref{cor:ShWa}.

\begin{remark} To facilitate following the next proof, let us note that
\begin{equation}\label{form96}
\min\{t,\frac{s + t}{2},1\} = \begin{cases} t, & t \leq s, \\ \frac{s + t}{2}, & t \in [s,2 - s],\\ 1, & t \in [2 - s,2]. \end{cases}
\end{equation}
\end{remark}

\begin{proof}[Proof of Theorem \ref{thm:Furstenberg-minimal-non-concentration}] Fix the parameters $s,t,u$ and $0 < \zeta < \tfrac{t}{2} + \tfrac{s \cdot u}{2}$ from the statement. Write
\begin{displaymath}
\omega:=\frac{t}{2}+\frac{s u}{2}-\zeta>0,
\end{displaymath}
and choose
\begin{displaymath}
0<\eta<\frac{\omega}{8}.
\end{displaymath}
Let $\xi = \xi(s,t,u,\zeta,\eta) > 0$ be another small parameter with
\begin{displaymath}
\xi<\frac{\omega}{8}.
\end{displaymath}
To avoid endpoint issues, we always apply Theorem \ref{thm:FurstIntro} with a strict margin $\eta$ below the corresponding endpoint value of $\gamma$. Set
\begin{displaymath}
K_{\eta}:=\{(s,t',\gamma'):\ t'\in[0,2],\ \eta\le \gamma' \le \min\{t',\frac{s+t'}{2},1\}-\eta\}.
\end{displaymath}
All parameter triples used below in applications of Theorem \ref{thm:FurstIntro} (namely for indices in $I_1^{\mathrm{big}}\cup I_2\cup I_3$) lie in $K_{\eta}$, and by the last sentence of Theorem \ref{thm:FurstIntro} we may choose $\epsilon_{\mathrm{F}}=\epsilon_{\mathrm{F}}(s,t,u,\zeta,\eta)>0$ and the small-scale threshold $\delta_0$ uniformly on this compact set. Let $\tau=\tau(2,\xi)>0$ be the constant from Lemma \ref{l:combinatorial-weak}, and choose
\begin{displaymath}
\epsilon=\epsilon(s,t,u,\zeta,\eta,\xi)>0
\end{displaymath}
so small that
\begin{displaymath}
\epsilon<\min\left\{\frac{\omega}{2s+4},\frac{\tau\epsilon_{\mathrm{F}}}{2}\right\}.
\end{displaymath}
In particular,
\begin{displaymath}
\frac{s\epsilon}{2}+2\eta+\frac{3\xi}{2}<\omega
\quad\text{and}\quad
\frac{2\epsilon}{\tau}<\epsilon_{\mathrm{F}}.
\end{displaymath}
Without loss of generality, we may assume that $\delta > 0$ has the form $\delta = \Delta^{m}$ for some $\Delta \sim_{\epsilon} 1$ and $m \in \N$, and that $\mathcal{P}$ is $\{\Delta^j\}_{j=1}^m$-uniform. This reduction is possible thanks to Lemma \ref{l:uniformization}, which in any case allows us to find a $\{\Delta^{j}\}_{j = 1}^{m}$-uniform subset $\mathcal{P}' \subset \mathcal{P}$ with $|\mathcal{P}'| \geq \delta^{\epsilon}|\mathcal{P}|$, provided that $\Delta = 2^{-T}$ with $T^{-1} \log (2T) \leq \epsilon$.

Let $\beta:[0,m]\to [0,2m]$
be the branching function of $\mathcal{P}$ defined in \S\ref{subsec:uniform-sets}. Apply Lemma \ref{l:combinatorial-weak} to $\beta$, with parameters $d = 2$ and $\xi$. Let $\{ a_j\}_{j = 0}^{n}$, $\{t_j\}_{j=0}^{n-1}$ be the resulting objects; then $n=O_{\xi}(1)$ by (i). Let $F$ be the piecewise affine function from Remark \ref{r:combinatorial-weak}. In particular,
\begin{equation}\label{eq:F-beta-loss}
|\beta(a_i)-\beta(a_j)-\big(F(a_i)-F(a_j)\big)|\le \xi m,\qquad i,j\in\{0,\ldots,n-1\}.
\end{equation}
Recall that $|\mathcal{P}|=\delta^{-t}$. Then the non-concentration assumption \eqref{eq:single-scale-non-conc} translates into
\begin{equation} \label{eq:minimal-non-conc}
\beta((1-t/2)m)\ge (u-\e) m.
\end{equation}
(The $\e m$ loss accounts for the fact that we had to extract a large uniform subset from the original $\mathcal{P}$.)

We will use the notation $\lessapprox_{\delta}$ to hide polylogarithmic losses.

We apply Corollary \ref{cor:ShWa} to the sequence $\Delta^{a_j}$, $j=n,n-1,\ldots,0$. With the notation of the corollary,
\begin{displaymath}
\frac{|\mathcal{T}|}{M}\gtrapprox_{\delta}  \prod_{j=0}^{n-1}  \frac{|\mathcal{T}_{\mathbf{p}_{j}}|}{M_j},
\end{displaymath}
where, for brevity, we write $M_j := M_{\mathbf{p}_j}$.
Motivated by \eqref{form96}, we split the indices $\{0,\ldots,n-1\}$ into 3 classes:
\begin{displaymath}
I_1 = \{ j: t_j \in [0,s]\},\quad  I_2 = \{ j: t_j \in (s,2-s)\},\quad  I_3 = \{ j: t_j \in [2-s,2]\}.
\end{displaymath}
Also split
\begin{displaymath}
I_1^{\mathrm{big}}:=\{j\in I_1:\ t_j\ge 2\eta\},\qquad
I_1^{\mathrm{small}}:=I_1\setminus I_1^{\mathrm{big}}.
\end{displaymath}
Recall that the $t_j$ are increasing. Let
\begin{displaymath}
[0,A_1] = \cup_{j\in I_1} [a_j,a_{j+1}], \quad [A_1,A_2]=\cup_{j\in I_2} [a_j,a_{j+1}], \quad [A_2,m] = \cup_{j\in I_3}  [a_j,a_{j+1}],
\end{displaymath}
with the convention that $A_1=0$ if $I_1=\emptyset$, $A_2=A_1$ if $I_2=\emptyset$, and $A_2=m$ if $I_3=\emptyset$ (it is easy to see that all these cases are consistent).

Next, we will estimate each factor $|\mathcal{T}_{\mathbf{p}_{j}}|/M_j$ individually. The idea is always the same: according to Corollary \ref{cor:ShWa} the tube family $\mathcal{T}_{\mathbf{p}_{j}}$ is associated with a  
\begin{displaymath}
(\Delta^{a_{j + 1} - a_{j}},s,C_{j},M_{j})\text{-nice configuration } (\mathcal{P}_{j},\mathcal{T}_{\mathbf{p}_{j}}),
\end{displaymath}
where $\mathcal{P}_{j}$ is a Frostman $(\Delta^{a_{j + 1} - a_{j}},t_{j})$-set (this follows from part (ii) of Lemma \ref{l:combinatorial-weak} and the branching function-to-Frostman dictionary, see e.g. \cite[Lemma 8.3]{OS23}). Now, the range of $t_{j}$ ($t_{j} \leq s$ or $t_{j} \in (s,2 - s)$ or $t_{j} \geq 2 - s$) determines the precise form of the lower bound produced by Theorem \ref{thm:FurstIntro} for $|\mathcal{T}_{\mathbf{p}_{j}}|/M_{j}$. 
Write $\Delta_j:=\Delta^{a_{j+1}-a_j}$. By Corollary \ref{cor:ShWa}, we have $C_j\lessapprox_{\delta}\delta^{-\epsilon}$. Since $a_{j+1}-a_j\ge \tau m$ by Lemma \ref{l:combinatorial-weak}(i), it follows that $\Delta_j\ge \delta^{\tau}$, and hence
\begin{equation}\label{eq:Cj-bound}
C_j\le \Delta_j^{-2\epsilon/\tau}\le \Delta_j^{-\epsilon_{\mathrm{F}}},
\end{equation}
provided that $\delta>0$ is chosen sufficiently small. Therefore all applications of Theorem \ref{thm:FurstIntro} below are legitimate with uniform constants depending only on $s,t,u,\zeta,\eta$. In particular, Theorem \ref{thm:FurstIntro} itself contributes no extra small losses in what follows; the only explicit losses come from the initial extraction of a large uniform subset and from comparing $F$ and $\beta$ via Remark \ref{r:combinatorial-weak}.

For the intervals $[\Delta^{a_{j+1}},\Delta^{a_j}]$ with $j\in I_1^{\mathrm{big}}$, thus $t_{j} \in [2\eta,s]$, we apply Theorem \ref{thm:FurstIntro} with
\begin{displaymath}
\gamma_j^{(1)}:=t_j-\eta<t_j,
\end{displaymath}
and for $j\in I_1^{\mathrm{small}}$ we use the trivial bound
\begin{displaymath}
\frac{|\mathcal{T}_{\mathbf{p}_{j}}|}{M_j}\gtrsim 1.
\end{displaymath}
Therefore, using Remark \ref{r:combinatorial-weak},
\begin{align}
\prod_{j\in I_1} \frac{|\mathcal{T}_{\mathbf{p}_{j}}|}{M_j} \notag
&\gtrapprox_{\delta} \prod_{j\in I_1^{\mathrm{big}}} \Delta^{(a_{j}-a_{j+1})(t_j-\eta)} \notag \\
&\gtrapprox_{\delta} \prod_{j\in I_1} \Delta^{(a_{j}-a_{j+1})(t_j-2\eta)_+} \notag \\
&\gtrapprox_{\delta} \Delta^{-F(A_1)+2\eta A_1} \notag \\
&\label{form20} \stackrel{\eqref{eq:F-beta-loss}}{\gtrapprox} \delta^{\xi}\Delta^{-\beta(A_1)+2\eta A_1}.
\end{align}
For the intervals $[\Delta^{a_{j+1}},\Delta^{a_j}]$ with $j\in I_2$, we apply Theorem \ref{thm:FurstIntro} with
\begin{displaymath}
\gamma_j^{(2)}:=\frac{s+t_j}{2}-\eta<\frac{s+t_j}{2},
\end{displaymath}
and Remark \ref{r:combinatorial-weak}, to get
\begin{align*}
\prod_{j\in I_2} \frac{|\mathcal{T}_{\mathbf{p}_{j}}|}{M_j}
&\gtrapprox_{\delta}
\prod_{j\in I_2}\Delta^{\frac{(a_{j}-a_{j+1})t_j}{2}}
\cdot \Delta^{\frac{s(a_{j}-a_{j+1})}{2}}
\cdot \Delta^{-\eta(a_{j+1}-a_j)} \\
&\gtrapprox_{\delta}
\Delta^{-\frac{F(A_2)-F(A_1)}{2}}
\cdot \Delta^{-\frac{s (A_2-A_1)}{2}}
\cdot \Delta^{\eta(A_2-A_1)} \\
&\stackrel{\eqref{eq:F-beta-loss}}{\gtrapprox}
\delta^{\xi/2}
\cdot \Delta^{-\frac{\beta(A_2)-\beta(A_1)}{2}}
\cdot \Delta^{-\frac{s (A_2-A_1)}{2}}
\cdot \Delta^{\eta(A_2-A_1)}.
\end{align*}
For the intervals $[\Delta^{a_{j+1}},\Delta^{a_j}]$ with $j\in I_3$, we apply Theorem \ref{thm:FurstIntro} with $\gamma^{(3)}:=1-\eta<1$ to get
\begin{equation} \label{eq:minimal-conc-FuRen}
\prod_{j\in I_3} \frac{|\mathcal{T}_{\mathbf{p}_{j}}|}{M_j}
\gtrapprox_{\delta}
\prod_{j\in I_3} \Delta^{(a_{j}-a_{j+1})(1-\eta)}
=\Delta^{A_2-m+\eta(m-A_2)}.
\end{equation}
Combining the three cases, and using that $\delta^{-\tfrac{t}{2}} = \Delta^{-\tfrac{1}{2}\beta(m)}$ and $\Delta^{\eta A_1}\ge \Delta^{\eta m}=\delta^{\eta}$,
\begin{align}
\frac{|\mathcal{T}|}{M}
&\gtrapprox_{\delta} \delta^{-\frac{t}{2}+2\eta+\frac{3\xi}{2}}
\left[
\Delta^{-\frac{\beta(A_1)}{2}} \cdot \Delta^{\frac{s(A_1-A_2)}{2}} \cdot \Delta^{(A_2-m) + \frac{1}{2}(\beta(m) - \beta(A_2))}
\right].\label{form135}
\end{align}
The inequality was written in this way to make explicit that all the powers of "$\Delta$" are non-positive. Let $A := (1 - \tfrac{t}{2})m$. We now divide the argument into cases, depending on the sizes of $A_{1},A_{2}$ relative to $A$.
\subsubsection*{Cases $A_{1} \geq A$ or $A_{2} \leq A$} The case $A_{1} \geq A$ is nearly trivial: in this case we only need to observe, by the monotonicity of $\beta$, that
\begin{displaymath}
\Delta^{-\frac{\beta(A_{1})}{2}} \geq \Delta^{-\frac{\beta(A)}{2}} \stackrel{\eqref{eq:minimal-non-conc}}{\geq} \delta^{-\frac{u-\epsilon}{2}}.
\end{displaymath}
Plugging this into \eqref{form135} leads to a better estimate than what we claimed.

Let us then consider the case $A_{2} \leq A$. Now the main observation is that since $\beta$ is $2$-Lipschitz, we have
\begin{align*}
\frac{1}{2}\beta(A_{2})
&\geq \frac{1}{2}\beta\left( A \right) - \left(A - A_{2} \right) \\
&\stackrel{\eqref{eq:minimal-non-conc}}{\geq}
\frac{(u-\epsilon)m}{2} - \left(A - A_{2} \right).
\end{align*}
Consequently, recalling that $\beta(m) = mt$, we get
\begin{align*}
\Delta^{(A_2-m) + \frac{1}{2}(\beta(m) - \beta(A_2))}
 &= \Delta^{A_{2} - A - \frac{1}{2}\beta(A_{2})} \\
&\geq \Delta^{-\frac{(u-\epsilon)m}{2}}
 = \delta^{-\frac{u-\epsilon}{2}}.
\end{align*}
Again, plugging this into \eqref{form135} leads to a better estimate than what we claimed. 

\subsubsection*{Case $A_{1} \leq A \leq A_{2}$} This is the main case. We split the factor in \eqref{form135} as
\begin{align*}
&\Delta^{-\frac{\beta(A_1)}{2}} \cdot \Delta^{\frac{s(A_1-A_2)}{2}} \cdot \Delta^{(A_2-m) + \frac{1}{2}(\beta(m) - \beta(A_2))} \\
&= \Delta^{-\frac{\beta(A_1)}{2}} \cdot \Delta^{\frac{s(A_{1} - A)}{2}} \cdot \Delta^{\frac{s(A - A_{2})}{2}} \cdot \Delta^{A_{2} - A - \frac{1}{2}\beta(A_{2})} \\
&= \Pi_{1} \cdot \Pi_{2} \cdot \Pi_{3} \cdot \Pi_{4}.
\end{align*}
It is desirable to combine the factors $\Pi_{3}$ and $\Pi_{4}$:
\begin{equation}\label{form120}
\Pi_{3} \cdot \Pi_{4} = \Delta^{(A_{2} - A)(1 - \frac{s}{2}) - \frac{1}{2}\beta(A_{2})}.
\end{equation}
We note that since $\beta(m) = mt$, and $\beta$ is $2$-Lipschitz, we have
\begin{displaymath}
\frac{1}{2}(mt - \beta(A_{2})) \leq m - A_{2} \quad \Longleftrightarrow \quad A_{2} - A \leq \frac{1}{2}\beta(A_{2}).
\end{displaymath}
Plugging this into \eqref{form120}, and using $A_{2} \geq A$, we find
\begin{align*}
\Pi_{3} \cdot \Pi_{4}
&= \Delta^{[(A_{2} - A)(1 - \frac{s}{2}) - (1 - \frac{s}{2})\frac{1}{2}\beta(A_{2})] - \frac{s}{4}\beta(A_{2})} \\
&\geq \Delta^{-\frac{\beta(A_{2})s}{4}}
\geq \Delta^{-\frac{\beta(A)s}{4}}
\stackrel{\eqref{eq:minimal-non-conc}}{\geq}
\Delta^{-\frac{ms(u-\epsilon)}{4}}
 = \delta^{-\frac{s(u-\epsilon)}{4}}.
\end{align*}
We can obtain a similar estimate for the product $\Pi_{1} \cdot \Pi_{2}$ by using the trivial estimate $\tfrac{1}{2}\beta(A_{1}) \geq \tfrac{1}{4}\beta(A_{1})s$, the $2$-Lipschitz property of $\beta$, and that $A_{1} \leq A$:
\begin{align*}
\Pi_{1} \cdot \Pi_{2}
&= \Delta^{-\frac{\beta(A_{1})}{2}} \cdot \Delta^{\frac{s(A_{1} - A)}{2}} \\
&\geq \Delta^{-\frac{\beta(A_{1})s}{4}} \cdot \Delta^{\frac{(\beta(A_{1}) - \beta(A))s}{4}} \\
&= \Delta^{-\frac{\beta(A)s}{4}}
\stackrel{\eqref{eq:minimal-non-conc}}{\geq}
\delta^{-\frac{s(u-\epsilon)}{4}}.
\end{align*}
All in all,
\begin{displaymath}
\Pi_{1} \cdot \Pi_{2} \cdot \Pi_{3} \cdot \Pi_{4} \geq \delta^{-\frac{s(u-\epsilon)}{2}}.
\end{displaymath}
Plugging this into \eqref{form135}, we obtain
\begin{displaymath}
\frac{|\mathcal{T}|}{M}\gtrapprox_{\delta}\delta^{-\Lambda},
\qquad
\Lambda:=\frac{t}{2}+\frac{s(u-\epsilon)}{2}-2\eta-\frac{3\xi}{2}.
\end{displaymath}
By the choices of $\eta,\epsilon,\xi$,
\begin{displaymath}
\Lambda
 =\frac{t}{2}+\frac{su}{2}-\Big(\frac{s\epsilon}{2}+2\eta+\frac{3\xi}{2}\Big)
 >\frac{t}{2}+\frac{su}{2}-\omega
 =\zeta.
\end{displaymath}
Therefore, for $\delta$ small enough (depending on $s,t,u,\zeta$), the claimed bound \eqref{form95} follows. \end{proof}

\section{Proof of Theorems \ref{thm1} and \ref{thm2}} 

\label{sec:proofs2}

We first prove Theorem \ref{thm1}. The proof of Theorem \ref{thm2} will only differ in the very final computations, and we will indicate the necessary changes. Fix $s \in (0,1]$, $t \in [s,2 - s]$, and $\eta > 0$. Let $P \subset [0,1]^{2}$ be a Katz-Tao $(\delta,t,\delta^{-\epsilon})$-set. The claim is that if $\delta,\epsilon > 0$ are sufficiently small in terms of $s,t,\eta$, then 
\begin{equation}\label{form7}
|\mathcal{T}| \geq \delta^{\eta}M|P|^{(s + t)/(2t)}.
\end{equation}

The "management of small constants" in these proofs is similar to that in the proof of Theorem \ref{thm:Furstenberg-minimal-non-concentration}. There are three main parameters: the given $\eta > 0$, then the $\epsilon > 0$ we will choose, and finally an intermediate parameter $\zeta \in (\epsilon,\eta)$, which is roughly the same as "$\epsilon_{F}$" in the proof of Theorem \ref{thm:Furstenberg-minimal-non-concentration}. Since adding in all the small constants (again) makes the argument look complicated, we opt here for a notationally lighter strategy: we allow the $\lessapprox_{\delta}$ notation to hide constants of the form $\delta^{-o_{\zeta \to 0}(1)}$. In this way, our final lower bound will have the form $\geq \delta^{o_{\zeta \to 0}(1)}M|P|^{(s + t)/(2t)}$, so we obtain \eqref{form7} by choosing $\zeta > 0$ small enough in terms of $\eta$. The required smallness of $\zeta$ will finally determine the threshold for $\epsilon \ll \zeta$, in fact via the relation \eqref{form18}.

%The $\gtrapprox_{\delta}$ notation will enable us to apply the Furstenberg set estimate, Theorem \ref{thm:FurstIntro}, at the endpoint $\gamma = \min\{t,(s + t)/2,1\}$ at the expense of replacing the lower bound $|\mathcal{T}|/M \geq \delta^{-\gamma}$ by 
%\begin{displaymath} |\mathcal{T}|/M \gtrapprox_{\delta} \delta^{-\min\{t,(s + t)/2,1\}}. \end{displaymath}

We then start the proof in earnest. First of all, we may assume that $P$ is a Katz-Tao $(\delta,t,1)$-set, since the Katz-Tao $(\delta,t,\delta^{-\epsilon})$-set $P$ contains a Katz-Tao $(\delta,t,1)$-subset $P'$ with cardinality $|P'| \gtrapprox_{\delta} |P|$, see \cite[Lemma 2.2]{2025arXiv251105091O}.

As in the proof of Theorem \ref{thm:Furstenberg-minimal-non-concentration}, we may assume that $\delta=\Delta^m$ and $P$ is $\{\Delta^j\}_{j=1}^m$-uniform for $\Delta\sim_{\epsilon}1$.  Let $\beta:[0,m]\to [0,2m]$ be the branching function of $P$, defined in Section \ref{subsec:uniform-sets}, and apply Lemma \ref{l:combinatorial-weak} to $\beta$, with parameters $d = 2$ and $\zeta$. We select $\epsilon > 0$ (depending on $\zeta$, therefore finally $\eta$) so small that 
\begin{equation}\label{form18}
2\epsilon/\tau(\zeta) \leq \zeta,
\end{equation}
where $\tau(\zeta) > 0$ is the parameter given by Lemma \ref{l:combinatorial-weak}. 

 Let $\{ a_j\}_{j = 0}^{n}$, $\{t_j\}_{j=0}^{n-1}$ be the objects given by Lemma \ref{l:combinatorial-weak}; then $n=O_{\zeta}(1)$ by Lemma \ref{l:combinatorial-weak}(i), and $a_{n} = m$. The hypothesis that $P$ is $(\delta,t,1)$-Katz-Tao implies the uniform upper bound
\begin{displaymath}
t_{j} \leq t,
\end{displaymath}
and in particular $t_j\le 2-s$ because $t\le 2-s$.

We apply Corollary \ref{cor:ShWa} to the sequence $\Delta^{a_j}$, $j=n,n-1,\ldots,0$. With the notation of the corollary, we bound
\begin{equation}\label{form8}
\frac{|\mathcal{T}|}{M}\gtrapprox_{\delta}  \prod_{j=0}^{n-1}  \frac{|\mathcal{T}_{\mathbf{p}_{j}}|}{M_j}.
\end{equation}
The implicit constant in the lower bound \eqref{form8} is of the form $\gtrsim (\log (1/\delta))^{-O_{\zeta}(1)}$, and in particular it is larger than $\delta^{\eta/2}$ if $\delta > 0$ is chosen sufficiently small.

Further in the notation of Corollary \ref{cor:ShWa}, $(S_{\mathbf{p}_{j}}(P \cap \mathbf{p}_{j}),\mathcal{T}_{\mathbf{p}_{j}})$ is a $(\Delta_{j},s,C_{\mathbf{p}_{j}},M_{j})$-nice configuration with $\Delta_{j} := \Delta^{a_{j + 1} - a_{j}}$, and $C_{\mathbf{p}_{j}} \lesssim (\log (1/\delta))^{O_{\epsilon}(1)}\delta^{-\epsilon}$. In particular, since $a_{j + 1} - a_{j} \geq \tau m$ according to Lemma \ref{l:combinatorial-weak}, and by \eqref{form18}, it holds 
\begin{equation}\label{form10}
C_{\mathbf{p}_{j}} \leq \Delta_{j}^{-2\epsilon/\tau} \leq \Delta_{j}^{-\zeta},
\end{equation}
provided that $\delta > 0$ was chosen sufficiently small.

We split the indices $\{0,\ldots,n - 1\}$ into $2$ classes:
\[
I_1 = \{ j \in \{0,\ldots,n - 1\} : t_j \in [0,s]\},\quad  I_2 = \{ j \in \{0,\ldots,n - 1\}: t_j \in (s,2 - s]\}.
\]
Recall that the numbers $t_j$ are increasing, and $a_{n} = m$. Let
\[
[0,\mathfrak{a}] = \cup_{j\in I_1} [a_j,a_{j+1}], \quad [\mathfrak{a},m]=\cup_{j\in I_2} [a_j,a_{j+1}],
\]
with the convention that $\mathfrak{a} = 0$ if $I_1=\emptyset$ and $\mathfrak{a} =m$ if $I_2=\emptyset$.

For $j\in I_1$, we apply Theorem \ref{thm:FurstIntro} to the configuration $(S_{\mathbf{p}_{j}}(P \cap \mathbf{p}_{j}),\mathcal{T}_{\mathbf{p}_{j}})$ at scale $\Delta_{j}$, and with $t_j \leq s$. Note that $S_{\mathbf{p}_{j}}(P \cap \mathbf{p}_{j})$ is $(\Delta_{j},t_{j},O_{\epsilon}(1))$-Frostman thanks to the $(t_{j},0)$-superlinearity of $(\beta,a_{j},a_{j + 1})$, see \cite[Lemma 8.3]{OS23} for the details. The conclusion is
\begin{equation}\label{form2}
\prod_{j\in I_1} \frac{|\mathcal{T}_{\mathbf{p}_{j}}|}{M_j} \gtrapprox_{\delta} \prod_{j\in I_1}  \Delta_{j}^{-t_{j}} \gtrapprox_{\delta} \Delta^{-\beta(\mathfrak{a})} = |P|_{\Delta^{\mathfrak{a}}}.
\end{equation}
In this bound it was crucial that the "niceness" constant of the configuration $(S_{\mathbf{p}_{j}}(P \cap \mathbf{p}_{j}),\mathcal{T}_{\mathbf{p}_{j}})$ is bounded from above by $\Delta_{j}^{-\zeta}$ thanks to \eqref{form10}. This guarantees that the implicit constant in \eqref{form2} is indeed of the form $\geq \delta^{o_{\zeta \to 0}(1)}$. Let us also mention that the middle estimate $\gtrapprox_{\delta} \Delta^{-\beta(\mathfrak{a})}$ in \eqref{form2} also uses Remark \ref{r:combinatorial-weak} in the same as the estimate \eqref{form20} in the proof of Theorem \ref{thm:Furstenberg-minimal-non-concentration}.

For $j\in I_2$, we similarly apply Theorem \ref{thm:FurstIntro} with $t_j \in [s,2 - s]$ and Remark \ref{r:combinatorial-weak}, to get
\begin{equation}\label{form3}
\prod_{j\in I_2} \frac{|\mathcal{T}_{\mathbf{p}_{j}}|}{M_j} \gtrapprox_{\delta} \prod_{j\in I_2} \Delta_{j}^{-(s + t_{j})/2} \gtrapprox_{\delta}  \Delta^{-\frac{\beta(m)-\beta(\mathfrak{a})}{2}} \Delta^{-\frac{s (m-\mathfrak{a})}{2}} = |P|_{\Delta^{\mathfrak{a}} \to \delta}^{1/2} \cdot \left(\frac{\Delta^{\mathfrak{a}}}{\delta} \right)^{s/2}.
\end{equation}
Here $|P|_{\Delta^{\mathfrak{a}} \to \delta} = |P|/|P|_{\Delta^{\mathfrak{a}}}$, and
\begin{displaymath}
\left(\frac{\Delta^{\mathfrak{a}}}{\delta} \right)^{s/2} \geq |P|_{\Delta^{\mathfrak{a}} \to \delta}^{s/(2t)}
\end{displaymath} 
thanks to the Katz-Tao $(\delta,t,1)$-set property of $P$. Combining the bounds from \eqref{form2}-\eqref{form3}, and noting that $1/2 + s/(2t) \leq 1$ by $s \leq t$, we obtain
\begin{displaymath}
\frac{|\mathcal{T}|}{M}\gtrapprox_{\delta} |P|_{\Delta^{\mathfrak{a}}}|P|_{\Delta^{\mathfrak{a}} \to \delta}^{1/2 + s/(2t)} \geq |P|_{\Delta^{\mathfrak{a}}}^{1/2 + s/(2t)}|P|_{\Delta^{\mathfrak{a}} \to \delta}^{1/2 + s/(2t)} = |P|^{(s + t)/(2t)}.
\end{displaymath} 
This completes the proof of Theorem \ref{thm1}.

We next discuss the modifications needed to prove Theorem \ref{thm2}. Let $A,B \subset \delta \Z \cap [0,1]$ as in the statement of Theorem \ref{thm2}. Then $P := A \times B$ satisfies the hypotheses of Theorem \ref{thm1} with $t := \alpha + \beta$, so we may repeat the proof to obtain the estimates \eqref{form2}-\eqref{form3}. Consequently,
\begin{equation}\label{form4}
\frac{|\mathcal{T}|}{M}\gtrapprox_{\delta} |P|_{\Delta^{\mathfrak{a}}}|P|_{\Delta^{\mathfrak{a}} \to \delta}^{1/2} \cdot \left(\frac{\Delta^{\mathfrak{a}}}{\delta} \right)^{s/2} = |A|_{\Delta^{\mathfrak{a}}}|B|_{\Delta^{\mathfrak{a}}}|A|_{\Delta^{\mathfrak{a}} \to \delta}^{1/2}|B|_{\Delta^{\mathfrak{a}} \to \delta}^{1/2} \cdot \left(\frac{\Delta^{\mathfrak{a}}}{\delta} \right)^{s/2}.
\end{equation} 
One minor difference, compared to previous argument, is that in the proof of Theorem \ref{thm1}, the given set $P$ was initially replaced by a $\{\Delta^{j}\}$-uniform subset of nearly comparable cardinality. Here we rather replace both $A$ and $B$ individually by $\{\Delta^{j}\}$-uniform subsets $A',B'$, and then define $P := A' \times B'$. Evidently $P$ is then also $\{\Delta^{j}\}$-uniform. Defining $P$ as a product set ensures that $|P|_{\Delta^{\mathfrak{a}}} = |A'|_{\Delta^{\mathfrak{a}}}|B'|_{\Delta^{\mathfrak{a}}}$ and $|P|_{\Delta^{\mathfrak{a}} \to \delta} = |A'|_{\Delta^{\mathfrak{a}} \to \delta}|B'|_{\Delta^{\mathfrak{a}} \to \delta}$, which was needed in \eqref{form4}.

To proceed, fix $\theta \in [0,1]$ satisfying 
\begin{equation}\label{form22}
\max\{\theta s/\alpha,(1 - \theta)s/\beta\} \leq 1.
\end{equation}
We will first obtain a $\theta$-dependent bound, and finally optimise $\theta$ to prove Theorem \ref{thm2}. By the Katz-Tao $(\delta,\alpha)$-set and $(\delta,\beta)$-set properties of $A$ and $B$,
\begin{displaymath}
\left(\frac{\Delta^{\mathfrak{a}}}{\delta} \right)^{s/2} = \left(\frac{\Delta^{\mathfrak{a}}}{\delta} \right)^{\theta s/2} \cdot \left(\frac{\Delta^{\mathfrak{a}}}{\delta} \right)^{(1 - \theta) s/2} \geq |A|_{\Delta^{\mathfrak{a}} \to \delta}^{\theta s/(2\alpha)}|B|_{\Delta^{\mathfrak{a}} \to \delta}^{(1 - \theta) s/(2\beta)}.
\end{displaymath} 
Plugging this lower bound into \eqref{form4}, and using 
\begin{displaymath}
1 \geq \max\{\frac{1}{2} + \theta s/(2\alpha),\frac{1}{2} + (1 - \theta)s/(2\beta)\},
\end{displaymath}
we conclude
\begin{align*} \frac{|\mathcal{T}|}{M} & \gtrapprox_{\delta} |A|_{\Delta^{\mathfrak{a}}}^{1/2 + \theta s/(2\alpha)}|B|_{\Delta^{\mathfrak{a}}}^{1/2 + (1 - \theta)s/(2\beta)}|A|_{\Delta^{\mathfrak{a}} \to \delta}^{1/2 + \theta s/(2\alpha)}|B|_{\Delta^{\mathfrak{a}} \to \delta}^{1/2 + (1 - \theta)s/(2\beta)}\\
& = (|A||B|)^{1/2}|B|^{s/(2\beta)} \cdot (|A|^{s/(2\alpha)}|B|^{-s/(2\beta)})^{\theta}. \end{align*} 
It remains to optimise the value of $\theta$ under the constraints $\theta \in [0,1]$ and \eqref{form22}. The assumption $|A|^{\beta} \geq |B|^{\alpha}$ is equivalent to $|A|^{s/(2\alpha)}|B|^{-s/(2\beta)} \geq 1$, so it is desirable to maximise $\theta$. 
\begin{remark} Actually the hypothesis $|A|^{\beta} \geq |B|^{\alpha}$ is not formally needed here -- or even in Theorem \ref{thm2} -- but it shows why we want to choose $\theta$ as large as possible. In the opposite case $|A|^{\beta} < |B|^{\alpha}$ we would instead minimise $\theta$ to obtain different bounds. We have not explicitly recorded these bounds anywhere, since they can be obtained by swapping the roles of $A,B$ in Theorem \ref{thm2}. \end{remark}
If $\alpha \leq s$, we choose $\theta := \alpha/s \in [0,1]$ to find
\begin{displaymath}
\frac{|\mathcal{T}|}{M} \gtrapprox_{\delta} |A||B|^{(\beta + s - \alpha)/(2\beta)}.
\end{displaymath}
(Note that the condition $(1 - \theta)s/\beta = (s - \alpha)/\beta \leq 1$ is guaranteed by $s \leq \alpha + \beta$.) If $\alpha \geq s$, we instead choose $\theta := 1$ to find
\begin{displaymath}
\frac{|\mathcal{T}|}{M} \gtrapprox_{\delta} |A|^{(\alpha + s)/(2\alpha)}|B|^{1/2}.
\end{displaymath}
This concludes the proof of Theorem \ref{thm2}.

\section{Sharpness of main results}

\label{sec:sharpness}

In this section we verify the sharpness of our main results. We start by recording the standard construction underlying the sharpness of the projection estimate corresponding to Theorem \ref{thm:FurstIntro}. This construction goes back at least to \cite{Wolff99}; see also \cite[\S2.2, Case 2]{FGR22} and \cite[Section A.1]{MR4745881}. In short, $P_{0}$ is a union of $\sim \rho^{-t}$ many $\rho$-balls arranged on a product grid, while $\Theta_{0}$ is obtained from arcs centred at rationals with small denominators.
\begin{lem}[Standard sharp projection example]\label{lem:standard-sharp-projection}
For every dyadic scale $\rho \in (0,1]$, every $s \in (0,1]$, and every $t \in [s,2-s]$, there exist:
\begin{itemize}
\item a set of the form $P_{0}=A\times A \subset \rho \Z^{2}\cap [0,1]^{2}$, where $A \subset \rho\Z\cap [0,1]$ is a Frostman and Katz-Tao $(\rho,t/2)$-set with $|A|\sim \rho^{-t/2}$,
\item a Frostman and Katz-Tao $(\rho,s)$-set $\Theta_{0}\subset S^{1}$ with $|\Theta_{0}|_{\rho}\sim \rho^{-s}$,
\end{itemize}
such that
\begin{displaymath}
|\pi_{\theta}(P_{0})|_{\rho} \lesssim \rho^{-(s+t)/2}, \qquad \theta \in \Theta_{0}.
\end{displaymath}
The Frostman and Katz-Tao constants in the claims above are absolute. Moreover, the elements of $P_{0}$ are $\gtrsim \rho^{t/2}$-separated.
\end{lem}

\subsection{Sharpness of Theorem \ref{thm:Furstenberg-minimal-non-concentration} and Corollary \ref{cor:projDeltaDiscretisedSingleScale}.}

\label{subsec:sharpness1}
 
It is enough to establish the sharpness of Corollary \ref{cor:projDeltaDiscretisedSingleScale}, since the corresponding sharpness for Theorem \ref{thm:Furstenberg-minimal-non-concentration} then follows by the standard projection/tube dictionary.

Fix parameters $s,t,u$ as in Corollary \ref{cor:projDeltaDiscretisedSingleScale}, and let $\delta>0$ be small. Choose dyadic scales $\Delta \sim \delta^{u}$ and $\rho \sim \delta^{1-t/2-u/2}$, so that
\begin{displaymath}
\delta^{1-t/2} \sim \rho\Delta^{1/2}
\quad\text{and}\quad
\rho^{2}\Delta\delta^{-2}\sim \delta^{-t}.
\end{displaymath}
Apply Lemma \ref{lem:standard-sharp-projection} with the parameter $t=1$ at scale $\Delta$. Thus there exist a set $\mathcal{P}_{1}=A_{1}\times A_{1} \subset \Delta^{1/2}\Z^{2}\cap [0,1]^{2}$ which is $(\Delta,1,O(1))$-Frostman, satisfies $|\mathcal{P}_{1}|_{\Delta}\sim \Delta^{-1}$, and, by the separation conclusion of the lemma, has the property that every $\Delta^{1/2}$-square meets $\mathcal{P}_{1}$ in $O(1)$ many $\Delta$-squares. Moreover, there exists a Frostman $(\Delta,s)$-set of slopes $\Theta_{1}$ such that
\begin{displaymath}
|\pi_{\theta}(\mathcal{P}_{1})|_{\Delta} \lesssim \Delta^{-(1+s)/2}, \qquad \theta \in \Theta_{1}.
\end{displaymath}

Now define $\mathcal{P}_{0}:=\rho\mathcal{P}_{1}$ (scaling by $\rho$). Then $|\mathcal{P}_{0}|_{\delta}\sim \delta^{-t}$, and because $\delta^{1-t/2}\sim \rho\Delta^{1/2}$, every $\delta^{1-t/2}$-square contains at most $O(\Delta)$-proportion of $\mathcal{P}_{0}$, i.e.
\begin{displaymath}
|\mathcal{P}_{0}\cap Q|_{\delta}\lesssim \delta^{u}|\mathcal{P}_{0}|_{\delta}, \qquad Q\in\mathcal{D}_{\delta^{1-t/2}}.
\end{displaymath}
Next modify $\Theta_{1}$ to a Frostman $(\delta,s,O(1))$-set $\Theta_{0}$ by replacing each $\Delta$-interval in $\Theta_{1}$ by $\sim (\Delta/\delta)^{s}$ many maximally spaced $\delta$-intervals. If $\theta'\in\Theta_{0}$ lies above $\theta\in\Theta_{1}$, then $|\theta'-\theta|\lesssim \Delta$, so
\begin{displaymath}
|\pi_{\theta'}(\mathcal{P}_{0})|_{\rho\Delta}\lesssim |\pi_{\theta}(\mathcal{P}_{0})|_{\rho\Delta}
\lesssim \Delta^{-(1+s)/2}.
\end{displaymath}
Passing from scale $\rho\Delta$ to scale $\delta$ gives, for every $\theta'\in\Theta_{0}$,
\begin{displaymath}
|\pi_{\theta'}(\mathcal{P}_{0})|_{\delta}
\lesssim \frac{\rho\Delta}{\delta}\cdot \Delta^{-(1+s)/2}
\sim \delta^{-(t/2+su/2)}.
\end{displaymath}
Since the same upper bound holds a fortiori for every subset $\mathcal{P}'\subset \mathcal{P}_{0}$, the exponent in Corollary \ref{cor:projDeltaDiscretisedSingleScale} is sharp.

\subsection{Sharpness of Theorem \ref{thm1}}

\label{subsec:sharpness2}

Fix $s \in (0,1]$, $t \in [s,2-s]$, and $\delta \in 2^{-\N}\cap(0,1]$. Choose $\rho \in 2^{-\N}\cap[\delta,1]$, and apply Lemma \ref{lem:standard-sharp-projection} at scale $\rho$. Then there exist a set $P_{0}$ and a Frostman $(\rho,s)$-set $\Theta_{0}$ as in the lemma, with $|P_{0}| \sim \rho^{-t}$ and
\begin{displaymath}
|\pi_{\theta}(P_{0})|_{\rho} \lesssim  \rho^{-(s+t)/2} \sim |P_{0}|^{(s+t)/(2t)}, \qquad \theta \in \Theta_{0}.
\end{displaymath}
Rescaling by the factor $\delta/\rho$, we obtain a Katz-Tao $(\delta,t,O(1))$-set $P := (\delta/\rho)P_{0}$ with $|P| = |P_{0}| \sim \rho^{-t}$ and
\begin{displaymath}
|\pi_{\theta}(P)|_{\delta} \lesssim \rho^{-(s+t)/2}\sim |P|^{(s+t)/(2t)}, \qquad \theta \in \Theta_{0}.
\end{displaymath}
Passing to the $\rho$-neighbourhood $\Theta$ of $\Theta_{0}$ preserves the same upper bound up to constants and yields a Frostman $(\delta,s)$-set of directions. Since $\rho$ can be chosen dyadically anywhere in $[\delta,1]$, the cardinality $|P| \sim \rho^{-t}$ ranges through the whole admissible interval $[1,\delta^{-t}]$ up to constants. This shows that the lower bound in Theorem \ref{thm1} is sharp throughout the admissible range of cardinalities.

\subsection{Sharpness of Theorem \ref{thm2}}\label{rem:sharpness-thm2} We finally discuss the sharpness of Theorem \ref{thm2} for product sets, but only in the special case $\alpha = \beta$, and $2\alpha \in [s,2 - s]$, as mentioned in Remark \ref{rem1}. In that case, recall that Theorem \ref{thm2} yields the bounds
\begin{equation}\label{form6}
|\mathcal{T}| \geq \delta^{\eta} M\begin{cases} |A||B|^{s/(2\alpha)}, & \alpha \leq s, \\ |A|^{1/2 + s/(2\alpha)}|B|^{1/2}, & \alpha \geq s, \end{cases}
\end{equation} 
for $|A| \geq |B|$ (this is also true but sub-optimal for $|A| < |B|$).

Let us first consider the sharpness of \eqref{form6} when $\alpha \leq s$. Fix the "target cardinalities" $|A| \geq |B|$. The aim is to construct Katz-Tao $(\delta,\alpha)$-sets $A,B \subset [0,1]$ with these cardinalities, and a family of tubes $\mathcal{T}$, such that $(A \times B,\mathcal{T})$ is a $(\delta,s,C,\delta^{-s})$-nice configuration, and the first bound in \eqref{form6} is attained.

Recall that the symmetric sharp example from Lemma \ref{lem:standard-sharp-projection} (with $t := 2\alpha$) has the form $B_{0}\times B_{0}$. Let $\ell:=\delta |B|^{1/\alpha}$, choose $B_{0}\subset \delta\Z\cap [0,\ell]$ to be a Katz-Tao $(\delta,\alpha)$-set with $|B_{0}|=|B|$, choose a Katz-Tao $(\ell,\alpha)$-set $D\subset \ell\Z\cap [0,1-\ell]$ with $|D|\sim N:=|A|/|B|$ (this is possible since $|A|\lesssim \delta^{-\alpha}$, hence $N\lesssim (\delta |B|^{1/\alpha})^{-\alpha}=\ell^{-\alpha}$), and define
\begin{displaymath}
A_{0}:=D+B_{0}.
\end{displaymath}
Then $|A_{0}|=N|B_{0}|\sim |A|$, and $A_{0}$ is again a Katz-Tao $(\delta,\alpha)$-set: if $I$ is an interval of length $r\le \ell$, it meets at most $O(1)$ translates $d+B_{0}$, while for $r\ge \ell$, it meets at most $O((r/\ell)^{\alpha})$ such translates, each contributing at most $|B_{0}|\sim (\ell/\delta)^{\alpha}$ points. Hence $A_{0}\times B_{0}$ is a union of $N$ disjoint translates of $B_{0}\times B_{0}$, and for the resulting configuration one has
\begin{displaymath}
|\mathcal{T}| \lesssim MN\cdot |B|^{(s + 2\alpha)/(2\alpha)} \sim M|A||B|^{s/(2\alpha)}.
\end{displaymath}
This matches the first bound in \eqref{form6}. 

We then discuss the sharpness of the second bound in \eqref{form6}. This is based on the following discrete statement:

\begin{lem}\label{lemma1} For $\mathbf{A},\mathbf{B},\mathbf{C} \in \N \, \setminus \, \{0\}$ satisfying 
\begin{equation}\label{form15}
\mathbf{A} \geq\{\mathbf{B},\mathbf{C}\} \quad \text{and} \quad \mathbf{A} \leq \mathbf{B}\mathbf{C},
\end{equation}
there exist sets $A,B,C \subset [0,1]$ of cardinalities comparable to $\mathbf{A},\mathbf{B},\mathbf{C}$, such that 
\begin{displaymath}
|A + cB| \lesssim \sqrt{\mathbf{A}\mathbf{B}\mathbf{C}}, \qquad c \in C.
\end{displaymath}
Moreover, the sets $A,B,C$ are maximally spaced:
\begin{equation}\label{form16}
|A|_{\mathbf{A}^{-1}} \sim \mathbf{A}, \quad |B|_{\mathbf{B}^{-1}} \sim \mathbf{B} \quad \text{and} \quad |C|_{\mathbf{C}^{-1}} \sim \mathbf{C}.
\end{equation}
\end{lem}

\begin{remark} This lemma is closely related to the "sharpness of optimal the $ABC$ sum-product theorem" in \cite[Theorem 1.2]{2024arXiv240912826L}. Since \cite{2024arXiv240912826L} does not appear to state the precise discrete claim useful for us, we give an elementary argument below. \end{remark} 

\begin{proof} Let $A \subset [0,1]$ be a maximal arithmetic progression (AP) containing $0$ with gap $\mathbf{A}^{-1}$. Next, let $B \subset A$ be a maximal sub-AP with $0 \in B$ and gap on the interval $[\tfrac{1}{2}\mathbf{B}^{-1},\mathbf{B}^{-1}]$. This exists, because $\mathbf{A}/\mathbf{B} \geq 1$, so there exists a natural number $m \in [\tfrac{1}{2}\mathbf{A}/\mathbf{B},\mathbf{A}/\mathbf{B}]$. With this notation the gap of $B$ could be $m \mathbf{A}^{-1}$.

Let $a \in A$ and $b \in B$ be the largest elements with 
\begin{displaymath}
a \leq \sqrt{\mathbf{C}/\mathbf{A}\mathbf{B}} \quad \text{and} \quad b \leq \sqrt{\mathbf{C}/\mathbf{A}\mathbf{B}},
\end{displaymath}
respectively. Note that $a \geq b > 0$ since $B \subset A$, and $\mathbf{B}^{-1} \leq \sqrt{\mathbf{C}/\mathbf{A}\mathbf{B}}$ thanks to \eqref{form15}. Using that $B$ is an AP, $b > 0$ implies that \emph{a fortiori} $a,b \sim \sqrt{\mathbf{C}/\mathbf{A}\mathbf{B}}$, and 
\begin{displaymath}
|B \cap [0,b]| \sim b\mathbf{B} \sim \sqrt{\mathbf{B}\mathbf{C}/\mathbf{A}}.
\end{displaymath}
Write
\begin{displaymath}
I := [0,a] \subset [0,\sqrt{\mathbf{C}/\mathbf{A}\mathbf{B}}] \stackrel{\eqref{form15}}{\subset} [0,1],
\end{displaymath}
and $A_{I} := A \cap I$, and $B_{I} := B \cap I = B \cap [0,b]$.

Let $C$ consist of rationals of the form $a/b \in [0,1]$, where 
\begin{displaymath}
a \in A_{I} - A_{I} \quad \text{and} \quad b \in (B_{I} - B_{I}) \, \setminus \, \{0\}.
\end{displaymath}
We claim that 
\begin{equation}\label{form24}
|C|_{\mathbf{C}^{-1}} \sim \mathbf{C}.
\end{equation} 
To see this, note first that $A' := \{\mathbf{A}a : a \in A \cap I\} = \{0,\ldots,\mathbf{A}a\}$ consists of all natural numbers between $0$ and $n := \mathbf{A}a$. Second, $B' := \{\mathbf{A}b : b \in B \cap I\} \subset \{0,\ldots,n\}$ is an arithmetic progression with $|B'| = |B_{I}| \sim \sqrt{\mathbf{B}\mathbf{C}/\mathbf{A}}$, and 
\begin{displaymath}
D := \mathrm{diam}(B') \sim \mathbf{A}b \sim \sqrt{\mathbf{A}\mathbf{C}/\mathbf{B}}.
\end{displaymath}
Now applying Corollary \ref{cor1} to $A'$ and $B'$ shows that
\begin{displaymath}
\Big| [0,1] \cap \frac{A' - A'}{(B' - B') \, \setminus \, \{0\}} \Big|_{(|B'|D)^{-1}} \sim |B'|D \sim \mathbf{C}.
\end{displaymath} 
This implies \eqref{form24}, since $A',B'$ are scaled copies of $A_{I},B_{I}$ with common factor $\mathbf{A}$.

Finally, we claim that
\begin{displaymath}
|A + cB| \lesssim \ell(I)|A||B| \sim \sqrt{\mathbf{A}\mathbf{B}\mathbf{C}}, \qquad c \in C,
\end{displaymath} 
where $\ell(I)$ is the length of $I$. To see this, fix $c = (a_{1} - a_{2})/(b_{1} - b_{2})$ and $r \in A + cB$, where $a_{j} \in A \cap I$ and $b_{j} \in B \cap I$. In particular $\max\{a_{j},b_{j}\} \leq \ell(I)$.

We will show that
\begin{equation}\label{form27}
|\{(a,b) \in (A + A - A) \times (B - B + B) : a + cb = r\}| \gtrsim \ell(I)^{-1}
\end{equation}
for all $r \in A + cB$, which yields the desired estimate
\begin{displaymath}
|A + cB| \lesssim \ell(I)|(A + A - A) \times (B - B + B)| \lesssim \ell(I)|A||B|.
\end{displaymath}
To prove \eqref{form27}, fix $r \in A + cB$, and let $(a_{0},b_{0}) \in A \times B$ with $a_{0} + cb_{0} = r$. Recalling that $c = (a_{1} - a_{2})/(b_{1} - b_{2})$, 
also
\begin{displaymath}
(a_{0} + ma_{1} - ma_{2}) + c(b_{0} - mb_{1} + mb_{2}) = r, \qquad m \in \Z.
\end{displaymath}
For $m \leq \ell(I)^{-1}$, it holds $\max\{ma_{j},mb_{j}\} \leq 1$. Thus, for $m \leq \ell(I)^{-1}$, the points $ma_{j}$ are elements of the AP $A \subset [0,1]$ (the maximal AP on $[0,1]$ with gap $\mathbf{A}^{-1}$) and the points $mb_{j}$ are similarly elements of $B \subset [0,1]$. Therefore $a_{0} + ma_{1} - ma_{2} \in A + A - A$ and $b_{0} - mb_{1} + mb_{2} \in B - B + B$ for $m \leq \ell(I)^{-1}$, giving \eqref{form27}.  \end{proof}

Now we get back to the sharpness of the second bound in \eqref{form6}. Let $|A|,|B| \in [1,\delta^{-\alpha}]$ with $|B| \leq |A|$ and $|B|^{\alpha} \geq |A|^{s - \alpha}$ be the "target cardinalities" for which we want to establish the sharpness of \eqref{form6}. The goal is to construct Katz-Tao $(\delta,\alpha)$-sets $A,B \subset [0,1]$ with roughly these cardinalities, and a Frostman $(\delta,s)$-set $\Theta \subset S^{1}$ such that $|\pi_{\theta}(A \times B)| \lesssim |A|^{1/2 + s/(2\alpha)}|B|^{1/2}$ for all $\theta \in \Theta$. To be precise, we will construct a Frostman $(\delta,s)$-set $C \subset [0,1]$ such that 
\begin{equation}\label{form25}
|A + cB|_{\delta} \lesssim |A|^{1/2 + s/(2\alpha)}|B|^{1/2}, \qquad c \in C.
\end{equation}

 Let $A_{0},B_{0},C_{0} \subset [0,1]$ be the sets from Lemma \ref{lemma1} with the choices 
\begin{displaymath}
\mathbf{A} := |A|, \quad \mathbf{B} := |B|, \quad \text{and} \quad \mathbf{C} := |A|^{s/\alpha}.
\end{displaymath}
Recalling that $s \leq \alpha$, one may check that the requirements \eqref{form15} are met; in particular $\mathbf{A} \leq \mathbf{B}\mathbf{C}$ is equivalent to $|B|^{\alpha} \geq |A|^{s - \alpha}$. Then \eqref{form16} promises that e.g. $C_{0}$ contains an $|A|^{-s/\alpha}$-separated subset of cardinality $\sim |C_{0}| \sim |A|^{s/\alpha}$. We reduce $C_{0}$ to this subset, so we assume in the sequel that $C_{0}$ is $|A|^{-s/\alpha}$-separated.

Write $\Delta := \delta |A|^{1/\alpha}$, and let
\begin{displaymath}
A := \Delta A_{0} \quad \text{and} \quad B := \Delta B_{0}.
\end{displaymath} 
Using $|B| \leq |A|$ and \eqref{form16}, it is easy to check that $A,B$ are Katz-Tao $(\delta,\alpha)$-sets. Moreover,
\begin{equation}\label{form17}
|A + cB|_{\delta} \leq |A_{0} + cB_{0}|_{|A|^{-1/\alpha}} \lesssim \sqrt{|A_{0}||B_{0}||C_{0}|} \sim |A|^{1/2 + s/(2\alpha)}|B|^{1/2}, \quad c \in C_{0}.
\end{equation} 
Evidently \eqref{form17} persists for all $c \in C$, where $C$ is the $|A|^{-1/\alpha}$-neighbourhood of $C_{0}$. Therefore we have established \eqref{form25} for $A,B$, and $C$. It remains is to check that $C$ is a Frostman $(\delta,s)$-set. 

First recall that $C_{0}$ is $|A|^{-s/\alpha}$-separated and $|C_{0}| \sim |A|^{s/\alpha}$. Since $s \leq 1$, it follows that
\begin{displaymath}
|C|_{\delta} \sim \frac{|A|^{-1/\alpha}}{\delta} \cdot |C_{0}| \sim |A|^{(s - 1)/\alpha}\delta^{-1}.
\end{displaymath}
Now, for $I \in \mathcal{D}_{r}([0,1])$ with $\delta \leq r \leq |A|^{-1/\alpha}$, we have $r^{1 - s} \leq |A|^{(s - 1)/\alpha}$, hence
\begin{displaymath}
|C \cap I|_{\delta} \leq (r/\delta) \leq r^{s}|A|^{(s - 1)/\alpha}\delta^{-1} = r^{s}|C|_{\delta}.
\end{displaymath}
For $|A|^{-1/\alpha} \leq r \leq |A|^{-s/\alpha}$ we use the $|A|^{-s/\alpha}$-separation of $C_{0}$ to estimate
\begin{displaymath}
|C \cap I|_{\delta} \lesssim \frac{|A|^{-1/\alpha}}{\delta} \leq r^{s}|A|^{(s - 1)/\alpha}\delta^{-1} \sim r^{s}|C|_{\delta}.
\end{displaymath} 
Finally, for $|A|^{-s/\alpha} \leq r \leq 1$, we use both the $|A|^{-s/\alpha}$-separation of $C_{0}$, and the Frostman $(|A|^{-s/\alpha},s)$-set property of $C_{0}$ to estimate
\begin{displaymath}
|C \cap I|_{\delta} \lesssim \frac{|A|^{-1/\alpha}}{\delta} \cdot |C_{0} \cap I| \leq \frac{|A|^{-1/\alpha}}{\delta} \cdot r^{s}|C_{0}| \sim r^{s}|C|_{\delta}.
\end{displaymath} 
We have now verified that $C$ is a Frostman $(\delta,s)$-set, and established the sharpness of the second bound in \eqref{form6}.

\appendix

\section{Ratios of arithmetic progressions}

This appendix contains some (with very high likelihood) standard facts which were needed to prove Lemma \ref{lemma1}. We start with the following variant of Dirichlet's approximation theorem:

\begin{lem}\label{lemmaDirichlet} Assume that $m,n \in \N \, \setminus \, \{0\}$ are such that $m \leq n$, and $m$ divides $n$. Write $A = \{0,1,\dots,n\}$ and $B = \{0,m,2m,\ldots,n\}$. Then, for all $x \in [0,1]$, there exist $a \in A - A$ and $b \in (B - B) \, \setminus \, \{0\}$ such that
\begin{displaymath}
|x - \frac{a}{b}| \leq \frac{m}{|b|n}.
\end{displaymath} 
\end{lem}

\begin{proof} Fix $x \in [0,1]$. For every $k \in \{0,\ldots,n/m\}$ we may write $kmx = n_{k} + x_{k}$, where $n_{k} \in \{0,\ldots,km\} \subset A$ and $x_{k} \in [0,1)$. There are $n/m + 1$ many numbers $x_{k}$, so there is a pair $x_{i},x_{j}$ with $i \neq j$, and $|x_{i} - x_{j}| \leq m/n$. Consequently,
\begin{displaymath}
|(i - j)mx - (n_{i} - n_{j})| = |x_{i} - x_{j}| \leq m/n,
\end{displaymath}
which can be rearranged to
\begin{displaymath}
\left|x - \frac{n_{i} - n_{j}}{im - jm} \right| \leq \frac{m}{|im - jm|n}.
\end{displaymath} 
This gives the claim with $a = n_{i} - n_{j}$ and $b = im - jm$.
\end{proof} 

The following proposition is an asymmetric variant of \cite[Claim A.5]{MR4745881}

\begin{prop}\label{prop1} Let $A,B$ be as in Lemma \ref{lemmaDirichlet}. Then 
\begin{displaymath}
\left|[0,1] \cap \frac{A - A}{(B - B) \, \setminus \, \{0\}} \right|_{(|A||B|)^{-1}} \gtrsim |A||B| = \frac{n^{2}}{m}.
\end{displaymath} 
\end{prop}

\begin{proof}

We rewrite the conclusion of Lemma \ref{lemmaDirichlet} as follows:
\begin{displaymath}
[0,1] \subset \bigcup_{b \in (B - B) \, \setminus \, \{0\}} \bigcup_{a \in A - A} B\left(\frac{a}{b},\frac{m}{|b|n}\right).
\end{displaymath}
We may sharpen this a little. If $b \in (B - B) \, \setminus \, \{0\}$ is fixed, all points $a \in A - A$ such that $|x - a/b| \leq (bn)^{-1}$ for some $x \in [0,1]$ satisfy $\mathrm{dist}(a,b[0,1]) \leq n^{-1}$. Thus $a$ lies on a certain interval $I_{b}$ of length at most $|b| + n^{-1}$. Taking this into account,
\begin{displaymath}
[0,1] \subset \bigcup_{b \in (B - B) \, \setminus \{0\}} \bigcup_{a \in (A - A) \cap I_{b}} B\left(\frac{a}{b},\frac{m}{|b|n}\right).
\end{displaymath}
Let us consider: how much of $[0,1]$ can be covered by that part of the previous union where $|b| \leq cn$? Here $c > 0$ is a suitable constant to be determined in a moment. Taking into account that $\mathrm{card}((A - A) \cap I_{b}) \leq |b| + 1 \leq 2|b|$ and $\mathrm{card}(\{b \in B - B : |b| \leq cn\}) \lesssim cn/m$, the measure of this "bad" part $I_{\mathrm{bad}} \subset [0,1]$ is at most
\begin{displaymath}
|I_{\mathrm{bad}}| \leq \mathop{\sum_{b \in (B - B) \, \setminus \, \{0\}}}_{|b| \leq cn} \mathrm{card}((A - A) \cap I_{b}) \cdot \frac{2m}{|b|n} \lesssim \mathop{\sum_{b \in (B - B) \, \setminus \, \{0\}}}_{|b| \leq cn} 2|b| \cdot \frac{2m}{|b|n} \lesssim c.   
\end{displaymath}
Thus, if $c > 0$ is small enough, we have $|[0,1] \, \setminus \, I_{\mathrm{bad}}| \geq \tfrac{1}{2}$. Moreover,
\begin{displaymath}
[0,1] \, \setminus \, I_{\mathrm{bad}} \subset \mathop{\bigcup_{b \in (B - B) \, \setminus \, \{0\}}}_{|b| \geq cn} \bigcup_{a \in (A - A)} B\left(\frac{a}{b},\frac{m}{|b|n} \right) \subset  \mathop{\bigcup_{b \in (B - B) \, \setminus \, \{0\}}}_{|b| \geq cn} \bigcup_{a \in (A - A)} B\left(\frac{a}{b},\frac{m}{cn^{2}} \right).
\end{displaymath} 
From this inclusion we see that
\begin{displaymath}
\frac{1}{2} \leq |[0,1] \, \setminus \, I_{\mathrm{bad}}| \lesssim |\{\frac{a}{b} : (a,b) \in (A - A) \times (B - B) \, \setminus \, \{0\}\}|_{m/n^{2}} \cdot \frac{m}{n^{2}},
\end{displaymath}
which gives the claim. \end{proof}

We finally record a more applicable version without divisibility hypotheses:

\begin{cor}\label{cor1} Let $n \in \N$, $A := \{0,\ldots,n\}$, and let $B \subset \{0,\ldots,n\}$ be an arithmetic progression with $D := \mathrm{diam}(B) > 0$. Then,
\begin{displaymath}
\left|[0,1] \cap \frac{A - A}{(B - B) \, \setminus \, \{0\}} \right|_{(|B|D)^{-1}} \gtrsim |B|D.
\end{displaymath} 
\end{cor}

\begin{proof} Write $B = \{b_{0},b_{0} + m_{0},b_{0} + 2m_{0},\ldots,b_{1}\}$, where $b_{1} - b_{0} = D$, and $m_{0} \geq 1$ is the gap of $B$. It suffices to show that
\begin{displaymath}
\left|[0,1] \cap \frac{(A \cap [b_{0},b_{1}]) - (A \cap [b_{0},b_{1}])}{(B - B) \, \setminus \, \{0\}} \right|_{(|B|D)^{-1}} \gtrsim |B|D
\end{displaymath}
Write $n_{0} := b_{1} - b_{0} \in \N \, \setminus \, \{0\}$, $A_{0} := (A \cap [b_{0},b_{1}]) - b_{0} = \{0,\ldots,n_{0}\}$, and $B_{0} := B - b_{0} \subset \{0,\ldots,n_{0}\}$. With this notation,
\begin{displaymath}
  \frac{(A \cap [b_{0},b_{1}]) - (A \cap [b_{0},b_{1}])}{(B - B) \, \setminus \, \{0\}} = \frac{A_{0} - A_{0}}{(B_{0} - B_{0}) \, \setminus \, \{0\}}.
\end{displaymath}  
Now the gap "$m_{0}$" of $B$ (hence $B_{0}$) divides $n_{0}$, so we may apply Proposition \ref{prop1}. Noting also that $|A_{0}| = D$ and $|B_{0}| = |B|$, the conclusion is
\begin{displaymath}
\left|[0,1] \cap \frac{A - A}{(B - B) \, \setminus \, \{0\}} \right|_{(|B|D)^{-1}} \geq \Big| [0,1] \cap \frac{A_{0} - A_{0}}{(B_{0} - B_{0}) \, \setminus \, \{0\}} \Big|_{(|B|D)^{-1}} \gtrsim |B|D,
\end{displaymath}  
as desired. \end{proof}

\def\cprime{$'$}

%\bibliographystyle{plain}
%\bibliography{references}

\end{document}